\documentclass{amsart}

\usepackage{amsfonts}
\usepackage{amssymb}
\usepackage{amsthm}
\usepackage{eucal}
\usepackage{tikz}

\theoremstyle{plain}
\newtheorem{theorem}{Theorem}[section]
\newtheorem{lemma}[theorem]{Lemma}

\newtheorem{corollary}[theorem]{Corollary}

\theoremstyle{definition}

\theoremstyle{remark}

\newcommand{\m}{\mathfrak{m}}
\newcommand{\J}{\mathfrak{J}}
\newcommand{\gr}{\mathrm{gr}}
\newcommand{\e}{\textbf{e}}
\newcommand{\1}{\textbf{1}}
\newcommand{\Z}{\mathbb{Z}}
\newcommand{\C}{\mathcal{C}}
\renewcommand{\P}{\mathcal{P}}
\newcommand{\Q}{\mathcal{Q}}

\begin{document}
\title[Jacobson graph]{Cycles and paths in Jacobson graphs}
\author{A. Azimi}
\author{M. Farrokhi D. G.}
\keywords{Jacobson graph, Hamiltonian cycle, Hamiltonian path, Eulerian tour, Eulerian trail, pancyclic}
\subjclass[2000]{Primary 05C45; Secondary 16P10, 13H99, 16N20.}
\address{Department of Pure Mathematics, Ferdowsi University of Mashhad, Mashhad, Iran}
\email{ali.azimi61@gmail.com}
\address{Department of Pure Mathematics, Ferdowsi University of Mashhad, Mashhad, Iran.}
\email{m.farrokhi.d.g@gmail.com}
\begin{abstract}
All finite Jacobson graphs with a Hamiltonian cycle or path, or Eulerian tour or trail are determined, and it is shown that a finite Jacobson graph is Hamiltonian if and only if it is pancyclic. Also, the length of the longest induced cycles and paths in finite Jacobson graphs are obtained.
\end{abstract}
\maketitle
\section{Introduction}
A \textit{Hamiltonian cycle (resp. path)} in a graph is a cycle (resp. path) including all the vertices of the graph, respectively. Similarly, an \textit{Eulerian tour or circuit (resp. trail)} in a graph is a closed walk (resp. walk) including all the edges of the graph, respectively. A graph is \textit{Hamiltonian} if it has a Hamiltonian cycle and it is \textit{Eulerian} if it has an Eulerian tour or circuit. Finding a cycle or path with a given property is usually a difficult problem and it is known to be an NP-complete problem in general, even for the special case of Hamiltonian cycles.

Let $R$ be a commutative ring with non-zero identity. The \textit{Jacobson graph} of $R$, denoted by $\J_R$, is a graph whose vertex set is $R\setminus J(R)$ such that two distinct vertices $x,y\in V(\J_R)$ are adjacent whenever $1-xy\not\in U(R)$, in which $U(R)$ is the group of units of $R$. Recall that the Jacobson radical $J(R)$ of $R$ is the intersection of all maximal ideals of $R$ and it has the property that $1-xr\in U(R)$ whenever $x\in J(R)$ and $r\in R$. 

The authors in \cite{aa-ae-mfdg} introduce the Jacobson graphs and study several graph theoretical properties of them. They prove that the Jacobson graph of a finite local ring $(R,\m)$ is a union of at most two complete graphs $K_{|\m|}$ together with some complete bipartite graphs $K_{|\m|,|\m|}$, while the Jacobson graph of a finite non-local ring is always connected with diameter at most $3$ and girth equal to $3$, except the Jacobson graph of the ring $\Z_2\oplus\Z_2$, which is a path of length two. Also, all finite rings with planar or perfect Jacobson graphs are determined and many numerical invariants of Jacobson graphs  including dominating number, independence number and vertex and edge chromatic numbers are computed.

The aim of this paper is to investigate some special cycles and paths in Jacobson graphs. In section 2, we will determine all Jacobson graphs with a Hamiltonian cycle or path, or Eulerian tour or trail. Moreover, it is shown that a Jacobson graph is Hamiltonian if and only if it contains cycles of all possible lengths. In section 3, we shall compute the length of the longest induced cycles and paths in Jacobson graphs. It is worth noting that induced cycles have many connections with other graph theoretical properties. For instance, the strong perfect graph theorem states that a graph $G$ is perfect if and only if neither $G$ nor its complement $\overline{G}$ contains an induced odd cycle of length $5$ or more. Also, Scott \cite{ads} gives a connection between induced cycles and the chromatic number of a graph. There are many works encountering the problem of finding the length of the longest induced cycles in several classes of graphs and we refer the reader to \cite{pa-sp,jla-pv,fb-fh,edf,yk} for details. Finally, we note that the graphs under consideration are good instances of the graphs studied in \cite{na,ye,vn-rhs}.

We begin with recalling some structural results on rings with non-zero identities. A finite ring $R$ is \textit{local} if it has only one maximal subgroup, say $\m$, for which the quotient ring $R/\m$ is a field. By \cite[Theorem VI.2]{brm}, a finite ring $R$ with a non-zero identity can be decomposed uniquely into a direct sum of finite local rings $R_1,\ldots,R_n$, that is
\[R=R_1\oplus\cdots\oplus R_n.\]
Is it known that
\[J(R)=J(R_1)\oplus\cdots\oplus J(R_n)\]
and
\[U(R)=U(R_1)\oplus\cdots\oplus U(R_n).\]
Utilizing the above notations, two distinct vertices $(x_1,\ldots,x_n)$ and $(y_1,\ldots,y_n)$ are adjacent if and only if $1-x_iy_i\notin U(R_i)$ or equivalently $1-x_iy_i\in J(R_i)$ for some $1\leq i\leq n$. We shall use this observation frequently in the proofs. For given such a decomposition, we use $\1$ for the element $(1,\ldots,1)$ with all entries equal to $1$ and $\e_i$ for the element with $1$ on its $i$th entry and $0$ elsewhere, for $i=1,\ldots,n$. Throughout this paper, all rings are commutative with a non-zero identity.
\section{Hamiltonian cycles and paths, and Eulerian tours and trails}
Let $R$ be a finite ring. If $R$ is a local ring with associated field $F$, then by the previous discussions, $\J_R$ is connected only if $|F|=2$, for which $\J_R$ is a complete graph. Hence, to find a Hamiltonian cycle or path in $\J_R$, we further assume that $R$ is a non-local ring. Note that the situation for Eulerian tours and trails is not as obvious as for Hamiltonian cycles and paths, when $R$ is a local ring. In fact, we show that $\Z_2$ is the only finite local ring which has an Eulerian tour, and $\Z_2$ and $\Z_4$ are the only finite local rings which have an Eulerian trail.

We begin with constructing Hamiltonian cycles and paths for Jacobson graphs of finite non-local rings.
\begin{theorem}\label{Hamiltonian}
Let $R$ be a finite non-local ring. Then $\J_R$ is Hamiltonian except when $R\cong\Z_2\oplus F$ ($F$ field), in which case $\J_R$ has a Hamiltonian path.
\end{theorem}
\begin{proof}
We first determine all finite semi-simple rings whose Jacobson graphs are Hamiltonian. First suppose that $S$ is a finite semi-simple ring with Hamiltonian Jacobson graph and $F$ is a field. Let $\C$ be a Hamiltonian cycle in $\J_S$ and $\P$ be the path obtained from $\C$ by removing the edge joining two adjacent vertices $x$ and $y$ in $\C$. Also, let $\Q$ be the sequence of elements of $S$ starting at $x$ and ending at $y$. Then $\J_{S\oplus F}$ is Hamiltonian and its Hamiltonian cycle is drawn in Figure 1, in which $F\setminus\{0,\pm1\}=\{a_1^{\pm1},\ldots,a_k^{\pm1}\}$ and $x\oplus y$ denotes the element $(x,y)$ for all $x\in S$ and $y\in F$. Now, we determine all Hamiltonian Jacobson graphs with minimal associated rings. Let $E$ and $F$ be two finite fields such that $|E|\geq|F|$. Let $T$ be an $|E|\times|F|$ grid and label the lattice points of $T\setminus\{(0,0)\}$ by elements of $E\times F$, where the elements of $E$ and $F$ are ordered in the following manner
\[0,a_1,a_1^{-1},\ldots,a_m,a_m^{-1},1,-1\]
and
\[0,b_1,b_1^{-1},\ldots,b_n,b_n^{-1},1,-1,\]
with $-1$ at the ends if $|E|$ and $|F|$ are odd, respectively.

If $|F|\geq4$, then $E\oplus F$ is Hamiltonian with a Hamiltonian cycle given in Figures 2, 3 and 4.

If $|F|=3$, then again $E\oplus F$ is Hamiltonian with a Hamiltonian cycle given in Figures 5 and 6.

Moreover, $K\oplus\Z_2\oplus\Z_2$ is Hamiltonian for every finite field $K$, and a Hamiltonian cycle is given in Figures 7 and 8. Here the elements of $K$ are ordered as same as the elements of $E$ and $F$, and elements of $\Z_2\oplus\Z_2$ is ordered as
\[(0,0),\ (0,1),\ (1,1),\ (1,0).\]

Now let $R=R_1\oplus\cdots\oplus R_n$ be a decomposition of a finite ring $R$ into local rings $R_i$ with associated fields $F_i$, respectively. A simple observation shows that $\J_R$ is Hamiltonian if $\J_{R/J(R)}$ is Hamiltonian. Hence, by previous discussions, $\J_R$ is Hamiltonian whenever $n\geq3$, or $n=2$ and $|F_1|,|F_2|\geq3$. Now suppose that $n=2$ and $|F_1|=2$. We first consider the case where $J(R)\neq0$. Let $\e'_2\in(\e_2+J(R))\setminus\{\e_2\}$ and $Q^\pm$ be the sequence of elements of $\pm\e'_2+J(R)$ starting at $\pm\e_2$ and ending at $\pm\e'_2$. Also, let $\P_i$ be the sequence of alternately chosen elements of $a_i^{-1}\e_2+J(R)$ and $a_i\e_2+J(R)$ starting at $a_i^{-1}\e_2$ and ending at $a_i\e_2$, for $i=1,\ldots,k$, where 
\[(R_2/\m_2)\setminus\{\m_2,\pm1+\m_2\}=\{a_1^{\pm1}+\m_2,\ldots,a_k^{\pm1}+\m_2\}.\]
If $\C$ is a subgraph obtained by removing the edges $\{\e_1+\e_2,\e_1+\e'_2\}$, $\{\e_1-\e_2,\e_1-\e'_2\}$ and $\{\e_1+a_i\e_2,\e_1+a_i^{-1}\e_2\}$ ($i=1,\ldots,k$) from a Hamiltonian cycle of the complete subgraph $\e_1+J(R)+0\oplus R_2$ containing the aforementioned edges, then the Figure 9 gives a Hamiltonian cycle for $R$. Hence $R$ is Hamiltonian. Finally, if $J(R)=0$, then $\J_R$ has a Hamiltonian path, which can be described in the same way as in the previous case. The proof is complete.
\end{proof}
\begin{center}
\begin{tikzpicture}[scale=0.6]
\draw[fill] (78:5cm) circle (2pt);
\node [rotate=78] at (78:6.2cm) {{\tiny $y\oplus a_{k-1}^{-1}$}};

\draw[fill] (66:5cm) circle (2pt);
\node [rotate=66] at (66:6.0cm) {{\tiny $x\oplus a_k$}};

\node [rotate=54] at (54:3cm) {{\tiny $\P\oplus a_k$}};

\draw[fill] (42:5cm) circle (2pt);
\node [rotate=42] at (42:6.0cm) {{\tiny $y\oplus a_k$}};

\draw[fill] (30:5cm) circle (2pt);
\node [rotate=30] at (30:6.0cm) {{\tiny $0\oplus a_k^{-1}$}};

\draw[fill] (18:5cm) circle (2pt);
\node [rotate=18] at (18:6.0cm) {{\tiny $0\oplus a_k$}};

\draw[fill] (6:5cm) circle (2pt);
\node [rotate=6] at (6:6.05cm) {{\tiny $x\oplus a_k^{-1}$}};

\node [rotate=-6] at (-6:3cm) {{\tiny $\P\oplus a_k^{-1}$}};

\draw[fill] (-18:5cm) circle (2pt);
\node [rotate=-18] at (-18:6.05cm) {{\tiny $y\oplus a_k^{-1}$}};

\draw[fill] (-42:5cm) circle (2pt);
\node [rotate=-42] at (-42:6.0cm) {{\tiny $x\oplus -1$}};

\node [rotate=-54] at (-54:3cm) {{\tiny $Q\oplus -1$}};

\draw[fill] (-66:5cm) circle (2pt);
\node [rotate=-66] at (-66:6.0cm) {{\tiny $y\oplus -1$}};

\draw[fill] (-78:5cm) circle (2pt);
\node [rotate=-78] at (-78:6.0cm) {{\tiny $x\oplus 0$}};

\node [rotate=-90] at (-90:3cm) {{\tiny $\P\oplus 0$}};

\draw[fill] (-102:5cm) circle (2pt);
\node [rotate=-102] at (-102:6.0cm) {{\tiny $y\oplus 0$}};

\draw[fill] (-114:5cm) circle (2pt);
\node [rotate=-114] at (-114:6.0cm) {{\tiny $x\oplus 1$}};

\node [rotate=-126] at (-126:3cm) {{\tiny $Q\oplus 1$}};

\draw[fill] (-138:5cm) circle (2pt);
\node [rotate=-138] at (-138:6.0cm) {{\tiny $y\oplus 1$}};

\draw[fill] (-162:5cm) circle (2pt);
\node [rotate=-162] at (-162:6.0cm) {{\tiny $x\oplus a_1$}};

\node [rotate=-174] at (-174:3cm) {{\tiny $\P\oplus a_1$}};

\draw[fill] (-186:5cm) circle (2pt);
\node [rotate=-186] at (-186:6.0cm) {{\tiny $y\oplus a_1$}};

\draw[fill] (-198:5cm) circle (2pt);
\node [rotate=-198] at (-198:6.05cm) {{\tiny $0\oplus a_1^{-1}$}};

\draw[fill] (-210:5cm) circle (2pt);
\node [rotate=-210] at (-210:6.0cm) {{\tiny $0\oplus a_1$}};

\draw[fill] (-222:5cm) circle (2pt);
\node [rotate=-222] at (-222:6.05cm) {{\tiny $x\oplus a_1^{-1}$}};

\node [rotate=-234] at (-234:3cm) {{\tiny $\P\oplus a_1^{-1}$}};

\draw[fill] (-246:5cm) circle (2pt);
\node [rotate=-246] at (-246:6.05cm) {{\tiny $y\oplus a_1^{-1}$}};

\draw[fill] (-258:5cm) circle (2pt);
\node [rotate=-258] at (-258:6.0cm) {{\tiny $x\oplus a_2$}};

\clip (0,0) circle (5cm);

\draw [thick] (78:5cm) arc (78:66:5cm);
\draw [densely dashed] (54:5cm) circle (1.045);
\draw [thick] (42:5cm) arc (42:6:5cm);
\draw [densely dashed] (-6:5cm) circle (1.045);
\draw [thick] (-18:5cm) arc (-18:-42:5cm);
\draw [densely dashed] (-54:5cm) circle (1.045);
\draw [thick] (-66:5cm) arc (-66:-78:5cm);
\draw [thick] (-102:5cm) arc (-102:-114:5cm);
\draw [densely dashed] (-90:5cm) circle (1.045);
\draw [densely dashed] (-126:5cm) circle (1.045);
\draw [thick] (-138:5cm) arc (-138:-162:5cm);
\draw [densely dashed] (-174:5cm) circle (1.045);
\draw [thick] (-186:5cm) arc (-186:-222:5cm);
\draw [densely dashed] (-234:5cm) circle (1.045);
\draw [thick] (-246:5cm) arc (-246:-258:5cm);
\draw [thick,dotted] (102:5cm) arc (102:78:5cm);
\end{tikzpicture}\\\vspace{-0.25cm}
Figure 1. $S\oplus F$
\end{center}
\begin{center}
\begin{tikzpicture}[scale=0.6]
\draw [dotted] (0,0) grid (4,4);
\draw [dotted] (5,0) grid (9,4);
\draw [dotted] (0,5) grid (4,9);
\draw [dotted] (5,5) grid (9,9);

\draw [thick] (4,3)--(3,2)--(2,1)--(1,0)--(2,0)--(1,1)--(0,2)--(0,1)--(1,2)--(2,3)--(3,4);
\draw [thick] (4,1)--(3,0)--(4,0)--(3,1)--(2,2)--(1,3)--(0,4)--(0,3)--(1,4);
\draw [thick] (2,4)--(3,3)--(4,2);

\draw [thick] (5,1)--(6,0)--(5,0)--(6,1)--(7,2)--(8,3)--(8,4);
\draw [thick] (5,3)--(6,2)--(7,1)--(8,0)--(7,0)--(8,1)--(8,2)--(7,3)--(6,4);
\draw [thick] (5,2)--(6,3)--(7,4);
\draw [thick] (9,0)--(9,4);

\draw [thick] (1,5)--(0,6)--(0,5)--(1,6)--(2,7)--(2,8)--(1,8)--(0,7)--(0,8)--(1,7)--(2,6)--(3,5);
\draw [thick] (2,5)--(3,6)--(4,7)--(4,8)--(3,8)--(3,7)--(4,6);
\draw [thick] (0,9)--(4,9);

\draw [thick] (5,6)--(6,7)--(6,8)--(5,8)--(5,7)--(6,6)--(7,5);
\draw [thick] (6,5)--(7,6)--(8,7)--(9,8);
\draw [thick] (8,5)--(8,6)--(7,7)--(7,8)--(8,9)--(9,9);
\draw [thick] (9,5)--(9,6)--(9,7)--(8,8)--(7,9)--(6,9)--(5,9);

\draw [thick] (9,0) to [out=77, in=-77] (9,8);
\draw [thick] (0,9) to [out=13, in=167] (9,9);

\draw [loosely dotted,thick] (4.2,0.5)--(4.8,0.5);
\draw [loosely dotted,thick] (4.2,1.5)--(4.8,1.5);
\draw [loosely dotted,thick] (4.2,2.5)--(4.8,2.5);
\draw [loosely dotted,thick] (4.2,3.5)--(4.8,3.5);

\draw [loosely dotted,thick] (4.2,5.5)--(4.8,5.5);
\draw [loosely dotted,thick] (4.2,6.5)--(4.8,6.5);
\draw [loosely dotted,thick] (4.2,7.5)--(4.8,7.5);
\draw [loosely dotted,thick] (4.2,8.5)--(4.8,8.5);

\draw [loosely dotted,thick] (0.5,4.2)--(0.5,4.8);
\draw [loosely dotted,thick] (1.5,4.2)--(1.5,4.8);
\draw [loosely dotted,thick] (2.5,4.2)--(2.5,4.8);
\draw [loosely dotted,thick] (3.5,4.2)--(3.5,4.8);

\draw [loosely dotted,thick] (5.5,4.2)--(5.5,4.8);
\draw [loosely dotted,thick] (6.5,4.2)--(6.5,4.8);
\draw [loosely dotted,thick] (7.5,4.2)--(7.5,4.8);
\draw [loosely dotted,thick] (8.5,4.2)--(8.5,4.8);

\draw [loosely dotted,thick] (4.2,4.2)--(4.82,4.82);
\draw [loosely dotted,thick] (4.82,4.2)--(4.2,4.82);
\end{tikzpicture}\\
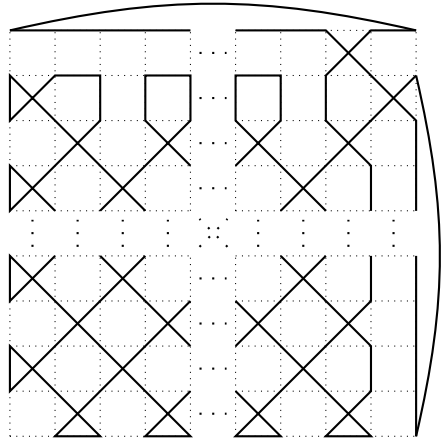
Figure 2. $E\oplus F$ ($|E|$ even, $|F|$ even)
\end{center}
\begin{center}
\begin{tikzpicture}[scale=0.6]
\draw [dotted] (0,0) grid (4,4);
\draw [dotted] (5,0) grid (9,4);
\draw [dotted] (0,5) grid (4,9);
\draw [dotted] (5,5) grid (9,9);

\draw [thick] (4,3)--(3,2)--(2,1)--(1,0)--(2,0)--(1,1)--(0,2)--(0,1)--(1,2)--(2,3)--(3,4);
\draw [thick] (4,1)--(3,0)--(4,0)--(3,1)--(2,2)--(1,3)--(0,4)--(0,3)--(1,4);
\draw [thick] (2,4)--(3,3)--(4,2);

\draw [thick] (5,1)--(6,0)--(5,0)--(6,1)--(7,2)--(8,3)--(8,4);
\draw [thick] (5,3)--(6,2)--(7,1)--(8,0)--(7,0)--(8,1)--(8,2)--(7,3)--(6,4);
\draw [thick] (5,2)--(6,3)--(7,4);
\draw [thick] (9,0)--(9,4);

\draw [thick] (2,5)--(1,6)--(0,7)--(0,6)--(1,7)--(2,7)--(2,6)--(1,5);
\draw [thick] (3,5)--(4,6)--(4,7)--(3,7)--(3,6)--(4,5);
\draw [thick] (0,8)--(4,8);
\draw [thick] (0,9)--(4,9);

\draw [thick] (5,5)--(6,6)--(6,7)--(5,7)--(5,6)--(6,5);
\draw [thick] (8,5)--(7,6)--(7,7)--(8,9)--(7,9)--(6,9)--(5,9);
\draw [thick] (9,5)--(9,6)--(8,7)--(7,8)--(6,8)--(5,8);
\draw [thick] (7,5)--(8,6)--(9,7);
\draw [thick] (9,9)--(9,8)--(8,8);

\draw [thick] (9,0) to [out=77, in=-77] (9,7);
\draw [thick] (0,9) to [out=13, in=167] (9,9);
\draw [thick] (0,8) to [out=13, in=167] (8,8);

\draw [loosely dotted,thick] (4.2,0.5)--(4.8,0.5);
\draw [loosely dotted,thick] (4.2,1.5)--(4.8,1.5);
\draw [loosely dotted,thick] (4.2,2.5)--(4.8,2.5);
\draw [loosely dotted,thick] (4.2,3.5)--(4.8,3.5);

\draw [loosely dotted,thick] (4.2,5.5)--(4.8,5.5);
\draw [loosely dotted,thick] (4.2,6.5)--(4.8,6.5);
\draw [loosely dotted,thick] (4.2,7.5)--(4.8,7.5);
\draw [loosely dotted,thick] (4.2,8.5)--(4.8,8.5);

\draw [loosely dotted,thick] (0.5,4.2)--(0.5,4.8);
\draw [loosely dotted,thick] (1.5,4.2)--(1.5,4.8);
\draw [loosely dotted,thick] (2.5,4.2)--(2.5,4.8);
\draw [loosely dotted,thick] (3.5,4.2)--(3.5,4.8);

\draw [loosely dotted,thick] (5.5,4.2)--(5.5,4.8);
\draw [loosely dotted,thick] (6.5,4.2)--(6.5,4.8);
\draw [loosely dotted,thick] (7.5,4.2)--(7.5,4.8);
\draw [loosely dotted,thick] (8.5,4.2)--(8.5,4.8);

\draw [loosely dotted,thick] (4.2,4.2)--(4.82,4.82);
\draw [loosely dotted,thick] (4.82,4.2)--(4.2,4.82);
\end{tikzpicture}\\
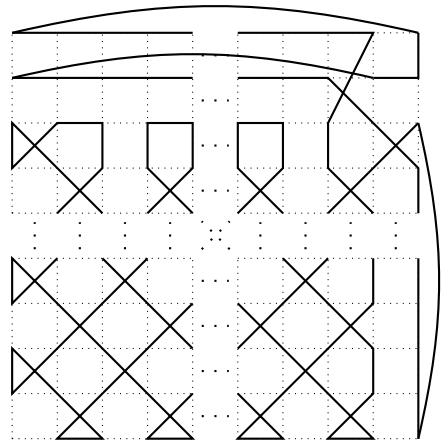
Figure 3. $E\oplus F$ ($|E|$ even, $|F|$ odd)
\end{center}
\begin{center}
\begin{tikzpicture}[scale=0.6]
\draw [dotted] (0,0) grid (4,4);
\draw [dotted] (5,0) grid (9,4);
\draw [dotted] (0,5) grid (4,9);
\draw [dotted] (5,5) grid (9,9);

\draw [thick] (4,3)--(3,2)--(2,1)--(1,0)--(2,0)--(1,1)--(0,2)--(0,1)--(1,2)--(2,3)--(3,4);
\draw [thick] (4,1)--(3,0)--(4,0)--(3,1)--(2,2)--(1,3)--(0,4)--(0,3)--(1,4);
\draw [thick] (2,4)--(3,3)--(4,2);

\draw [thick] (5,2)--(6,1)--(7,0)--(6,0)--(7,1)--(7,2)--(6,3)--(5,4);
\draw [thick] (5,1)--(6,2)--(7,3)--(7,4);
\draw [thick] (5,3)--(6,4);
\draw [thick] (8,0)--(8,4);
\draw [thick] (9,0)--(9,4);

\draw [thick] (1,5)--(2,6)--(2,7)--(1,7)--(0,6)--(0,7)--(1,6)--(2,5);
\draw [thick] (3,5)--(4,6)--(4,7)--(3,7)--(3,6)--(4,5);
\draw [thick] (0,8)--(4,8);
\draw [thick] (0,9)--(4,9);

\draw [thick] (6,5)--(7,6)--(9,7);
\draw [thick] (7,5)--(6,6)--(6,7)--(7,9)--(5,9);
\draw [thick] (8,5)--(8,6)--(7,7)--(6,8);
\draw [thick] (9,5)--(9,7);
\draw [thick] (5,8)--(6,8);
\draw [thick] (9,8)--(7,8);
\draw [thick] (8,9)--(9,9);
\draw [thick] (5,6)--(5,7);

\draw [thick] (9,0) to [out=77, in=-77] (9,8);
\draw [thick] (0,9) to [out=13, in=167] (9,9);
\draw [thick] (0,8) to [out=13, in=167] (7,8);
\draw [thick] (8,0) to [out=77, in=-77] (8,7);
\draw [thick] (8,7) to [out=77, in=-77] (8,9);

\draw [loosely dotted,thick] (4.2,0.5)--(4.8,0.5);
\draw [loosely dotted,thick] (4.2,1.5)--(4.8,1.5);
\draw [loosely dotted,thick] (4.2,2.5)--(4.8,2.5);
\draw [loosely dotted,thick] (4.2,3.5)--(4.8,3.5);

\draw [loosely dotted,thick] (4.2,5.5)--(4.8,5.5);
\draw [loosely dotted,thick] (4.2,6.5)--(4.8,6.5);
\draw [loosely dotted,thick] (4.2,7.5)--(4.8,7.5);
\draw [loosely dotted,thick] (4.2,8.5)--(4.8,8.5);

\draw [loosely dotted,thick] (0.5,4.2)--(0.5,4.8);
\draw [loosely dotted,thick] (1.5,4.2)--(1.5,4.8);
\draw [loosely dotted,thick] (2.5,4.2)--(2.5,4.8);
\draw [loosely dotted,thick] (3.5,4.2)--(3.5,4.8);

\draw [loosely dotted,thick] (5.5,4.2)--(5.5,4.8);
\draw [loosely dotted,thick] (6.5,4.2)--(6.5,4.8);
\draw [loosely dotted,thick] (7.5,4.2)--(7.5,4.8);
\draw [loosely dotted,thick] (8.5,4.2)--(8.5,4.8);

\draw [loosely dotted,thick] (4.2,4.2)--(4.82,4.82);
\draw [loosely dotted,thick] (4.82,4.2)--(4.2,4.82);
\end{tikzpicture}\\
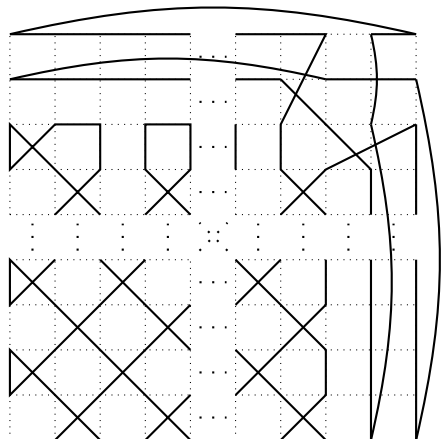
Figure 4. $E\oplus F$ ($|E|$ odd, $|F|$ odd)
\end{center}
\begin{center}
\begin{tikzpicture}[scale=0.6]
\draw [dotted] (0,0) grid (4,2);
\draw [dotted] (5,0) grid (9,2);

\draw [thick] (4,1)--(3,0)--(4,0)--(3,1)--(2,1)--(1,0)--(2,0)--(1,1)--(0,1);
\draw [thick] (5,1)--(6,0)--(5,0)--(6,1)--(7,1)--(8,0)--(7,0)--(8,1)--(7,2)--(5,2);
\draw [thick] (0,2)--(4,2);
\draw [thick] (8,2)--(9,2);
\draw [thick] (9,0)--(9,1);

\draw [thick] (0,1) to [out=13, in=167] (9,1);
\draw [thick] (0,2) to [out=13, in=167] (8,2);
\draw [thick] (9,0) to [out=70, in=-70] (9,2);

\draw [loosely dotted,thick] (4.2,0.5)--(4.8,0.5);
\draw [loosely dotted,thick] (4.2,1.5)--(4.8,1.5);
\end{tikzpicture}\\
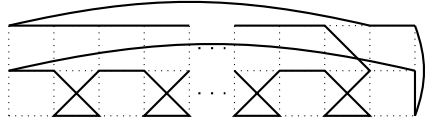
Figure 5. $E\oplus\Z_3$ ($|E|$ even)
\end{center}
\begin{center}
\begin{tikzpicture}[scale=0.6]
\draw [dotted] (0,0) grid (4,2);
\draw [dotted] (5,0) grid (9,2);

\draw [thick] (4,1)--(3,0)--(4,0)--(3,1)--(2,1)--(1,0)--(2,0)--(1,1)--(0,1);
\draw [thick] (5,1)--(6,1)--(7,0)--(6,0)--(7,1)--(8,1)--(8,0);
\draw [thick] (9,0)--(9,1);
\draw [thick] (0,2)--(4,2);
\draw [thick] (5,2)--(8,2);

\draw [thick] (0,1) to [out=13, in=167] (9,1);
\draw [thick] (0,2) to [out=13, in=167] (9,2);
\draw [thick] (8,0) to [out=70, in=-70] (8,2);
\draw [thick] (9,0) to [out=70, in=-70] (9,2);

\draw [loosely dotted,thick] (4.2,0.5)--(4.8,0.5);
\draw [loosely dotted,thick] (4.2,1.5)--(4.8,1.5);
\end{tikzpicture}\\
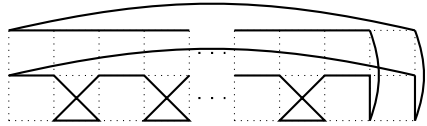
Figure 6. $E\oplus\Z_3$ ($|E|$ odd)
\end{center}
\begin{center}
\begin{tikzpicture}[scale=0.6]
\draw [dotted] (0,0) grid (4,3);
\draw [dotted] (5,0) grid (9,3);

\draw [thick](4,1)--(3,0)--(4,0)--(3,1)--(2,1)--(1,0)--(2,0)--(1,1)--(0,1)--(0,3)--(1,3)--(1,2)--(2,2)--(2,3)--(3,3)--(3,2)--(4,2)--(4,3);
\draw [thick] (5,1)--(6,0)--(5,0)--(6,1)--(7,1)--(8,0)--(7,0)--(8,1)--(9,1)--(9,0);

\draw [thick] (5,3)--(5,2)--(6,2)--(6,3)--(7,3)--(7,2)--(8,2)--(8,3)--(9,3)--(9,2);

\draw [thick] (9,0) to [out=70, in=-70] (9,2);

\draw [loosely dotted,thick] (4.2,0.5)--(4.8,0.5);
\draw [loosely dotted,thick] (4.2,1.5)--(4.8,1.5);
\draw [loosely dotted,thick] (4.2,2.5)--(4.8,2.5);
\end{tikzpicture}\\
Figure 7. $K\oplus\Z_2\oplus\Z_2$ ($|K|$ even)
\end{center}
\begin{center}
\begin{tikzpicture}[scale=0.6]
\draw [dotted] (0,0) grid (4,3);
\draw [dotted] (5,0) grid (9,3);

\draw [thick](4,1)--(3,0)--(4,0)--(3,1)--(2,1)--(1,0)--(2,0)--(1,1)--(0,1)--(0,3)--(1,3)--(1,2)--(2,2)--(2,3)--(3,3)--(3,2)--(4,2)--(4,3);
\draw [thick] (5,1)--(6,1)--(7,0)--(6,0)--(7,1)--(8,1)--(8,0);
\draw [thick] (8,2)--(9,2)--(9,0);
\draw [thick] (9,3)--(7,3)--(7,2)--(6,2)--(6,3)--(5,3)--(5,2);

\draw [thick] (8,0) to [out=70, in=-70] (8,2);
\draw [thick] (9,0) to [out=70, in=-70] (9,3);

\draw [loosely dotted,thick] (4.2,0.5)--(4.8,0.5);
\draw [loosely dotted,thick] (4.2,1.5)--(4.8,1.5);
\draw [loosely dotted,thick] (4.2,2.5)--(4.8,2.5);
\end{tikzpicture}\\
Figure 8. $K\oplus\Z_2\oplus\Z_2$ ($|K|$ odd)
\end{center}
\begin{center}
\begin{tikzpicture}[scale=0.6]
\draw [fill] (65:4cm) circle (2pt);
\node [rotate=65] at (65:2.5cm) {{\tiny $\e_1+a_{k-1}^{-1}\e_2$}};
\draw [fill] (50:4cm) circle (2pt);
\node [rotate=50] at (50:2.8cm) {{\tiny $\e_1+a_k\e_2$}};
\draw [fill] (50:5cm) circle (2pt);
\node [rotate=50-90] at (60:5cm) {{\tiny $a_k^{-1}\e_2$}};
\draw [fill] (20:6cm) circle (2pt);
\node [rotate=20] at (20:2.7cm) {{\tiny $\e_1+a_k^{-1}\e_2$}};
\draw [fill] (20:4cm) circle (2pt);
\node [rotate=20-90] at (12:6cm) {{\tiny $a_k\e_2$}};
\draw [fill] (-55:4cm) circle (2pt);
\node [rotate=-55] at (-55:3cm) {{\tiny $\e_1-\e_2$}};
\draw [fill] (-55:4.5cm) circle (2pt);
\node [rotate=-55] at (-55:5.1cm) {{\tiny $-\e_2$}};
\draw [fill] (-65:4.5cm) circle (2pt);
\node [rotate=-65] at (-65:3cm) {{\tiny $\e_1-\e'_2$}};
\draw [fill] (-65:4cm) circle (2pt);
\node [rotate=-65] at (-65:5.1cm) {{\tiny $-\e'_2$}};
\draw [fill] (-115:4cm) circle (2pt);
\node [rotate=-115] at (-115:5.1cm) {{\tiny $\e_2$}};
\draw [fill] (-115:4.5cm) circle (2pt);
\node [rotate=-115] at (-115:3cm) {{\tiny $\e_1+\e_2$}};
\draw [fill] (-125:4.5cm) circle (2pt);
\node [rotate=-125] at (-125:5.1cm) {{\tiny $\e'_2$}};
\draw [fill] (-125:4cm) circle (2pt);
\node [rotate=-125] at (-125:3cm) {{\tiny $\e_1+\e'_2$}};
\draw [fill] (160:4cm) circle (2pt);
\node [rotate=160] at (160:2.8cm) {{\tiny $\e_1+a_1\e_2$}};
\draw [fill] (160:6cm) circle (2pt);
\node [rotate=160-90] at (168:6cm) {{\tiny $a_1^{-1}\e_2$}};
\draw [fill] (130:5cm) circle (2pt);
\node [rotate=130] at (130:2.7cm) {{\tiny $\e_1+a_1^{-1}\e_2$}};
\draw [fill] (130:4cm) circle (2pt);
\node [rotate=130-90] at (122:5	cm) {{\tiny $a_1\e_2$}};
\draw [fill] (115:4cm) circle (2pt);
\node [rotate=115] at (115:2.8cm) {{\tiny $\e_1+a_2\e_2$}};
\node [rotate=0] at (0:0cm) {{\tiny $\C$}};

\draw [dashed] (-65:4cm) arc (-65:-115:4cm);

\draw [thick] (-65:4cm)--(-65:4.5cm);
\draw [thick] (-55:4cm)--(-55:4.5cm);
\draw [dashed] (-55:4.5) arc (96.9085:-216.9085:1cm);
\node [rotate=-55] at (-60:7cm) {{\tiny $\Q^-$}};

\draw [thick] (-115:4cm)--(-115:4.5cm);
\draw [thick] (-125:4cm)--(-125:4.5cm);
\draw [dashed] (-115:4.5) arc (36.9085:-276.9085:1cm);
\node [rotate=-120] at (-120:7cm) {{\tiny $\Q^+$}};

\draw [help lines,dash pattern=on 1pt off 1pt on 2pt off 1pt,rotate=20] (0:5.5cm) ellipse  (1cm and 0.65cm);
\draw [help lines,dash pattern=on 1pt off 1pt on 2pt off 1pt,rotate=50] (0:5.5cm) ellipse  (1cm and 0.65cm);
\draw [dashed] (-55:4cm) arc (-55:20:4cm);
\draw [thick] (20:4cm)--(20:6cm);
\draw [thick] (50:4cm)--(50:5cm);
\draw [thick] (50:4cm) arc (50:65:4cm);

\draw [help lines,dash pattern=on 1pt off 1pt on 2pt off 1pt,rotate=130] (0:5.5cm) ellipse  (1cm and 0.65cm);
\draw [help lines,dash pattern=on 1pt off 1pt on 2pt off 1pt,rotate=160] (0:5.5cm) ellipse  (1cm and 0.65cm);
\draw [dashed] (235:4cm) arc (235:160:4cm);
\draw [thick] (160:4cm)--(160:6cm);
\draw [thick] (130:4cm)--(130:5cm);
\draw [thick] (130:4cm) arc (130:115:4cm);

\draw [thick,dotted] (65:4cm) arc (65:115:4cm);

\draw [dashed] (20:6cm)--(47:5.9cm)--(23:5cm)--(50:5cm);
\node [rotate=35] at (34:6.3cm) {{\tiny $\P_k$}};
\draw [dashed] (160:6cm)--(133:5.9cm)--(157:5cm)--(130:5cm);
\node [rotate=145] at (145:6.3cm) {{\tiny $\P_1$}};
\end{tikzpicture}\\
Figure 9. $R_1\oplus R_2$ ($J(R)\neq0$)
\end{center}

Utilizing Theorem \ref{Hamiltonian}, we can prove a stronger result on the cycles of Jacobson graphs. Recall that a graph $G$ is \textit{pancyclic} if it contains cycles of lengths $l$ for each $3\leq l\leq|V(G)|$.
\begin{theorem}\label{pancyclic}
Let $R$ be a finite non-local ring. If $R\not\cong\Z_2\oplus F$ ($F$ field), then $\J_R$ is pancyclic.
\end{theorem}
\begin{proof}
If $R=S\oplus F$, where $S$ is a finite semi-simple ring with Hamiltonian Jacobson graph and $F$ is a finite field, then by adding suitable edges of types
\[\{x\oplus a_i^{\pm1},y\oplus a_i^{\pm1}\},\ \{x\oplus a_i^{\pm1},y\oplus a_i^{\mp1}\},\ \{x\oplus a_i^{\pm1},y\oplus1\},\ \{x\oplus1,y\oplus a_i^{\pm1}\}\]
to the cycle in Figure 1, and removing some vertices from the complete subgraph $Q\oplus1$ if necessary, we can construct cycles of arbitrary lengths $l$, for $3\leq l\leq|V(\J_R)|$. Hence $\J_R$ is pancyclic.

If $R=E\oplus F$ ($|E|,|F|>2$) or $E\oplus\Z_2\oplus\Z_2$, where $E$ and $F$ are finite fields, then by starting with Hamiltonian cycles given in Figures 2, 3, 4, 5, 6, 7 and 8, and complementing paths of length two with vertices in the same row or column, or by applying the following substitutions of edges
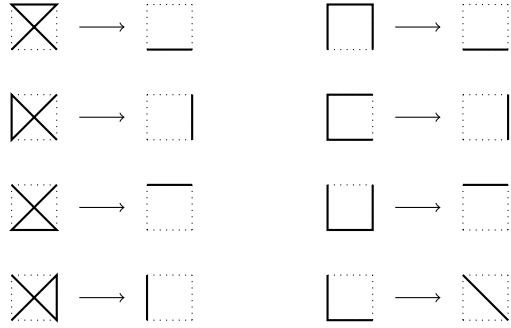
\begin{center}
\begin{tikzpicture}[scale=0.6]
\draw [dotted] (0,0) grid (1,1);
\draw [dotted] (3,0) grid (4,1);

\draw [dotted] (7,0) grid (8,1);
\draw [dotted] (10,0) grid (11,1);

\draw [dotted] (0,2) grid (1,3);
\draw [dotted] (3,2) grid (4,3);

\draw [dotted] (7,2) grid (8,3);
\draw [dotted] (10,2) grid (11,3);

\draw [dotted] (0,4) grid (1,5);
\draw [dotted] (3,4) grid (4,5);

\draw [dotted] (7,4) grid (8,5);
\draw [dotted] (10,4) grid (11,5);

\draw [dotted] (0,6) grid (1,7);
\draw [dotted] (3,6) grid (4,7);

\draw [dotted] (7,6) grid (8,7);
\draw [dotted] (10,6) grid (11,7);

\draw [thick] (0,0)--(1,1)--(1,0)--(0,1);
\draw [->] (1.5,0.5)--(2.5,0.5);
\draw [thick] (3,0)--(3,1);

\draw [thick] (8,0)--(7,0)--(7,1);
\draw [->] (8.5,0.5)--(9.5,0.5);
\draw [thick] (10,1)--(11,0);

\draw [thick] (0,3)--(1,2)--(0,2)--(1,3);
\draw [->] (1.5,2.5)--(2.5,2.5);
\draw [thick] (3,3)--(4,3);

\draw [thick] (7,3)--(7,2)--(8,2)--(8,3);
\draw [->] (8.5,2.5)--(9.5,2.5);
\draw [thick] (10,3)--(11,3);

\draw [thick] (1,5)--(0,4)--(0,5)--(1,4);
\draw [->] (1.5,4.5)--(2.5,4.5);
\draw [thick] (4,4)--(4,5);

\draw [thick] (8,4)--(7,4)--(7,5)--(8,5);
\draw [->] (8.5,4.5)--(9.5,4.5);
\draw [thick] (11,4)--(11,5);

\draw [thick] (0,6)--(1,7)--(0,7)--(1,6);
\draw [->] (1.5,6.5)--(2.5,6.5);
\draw [thick] (3,6)--(4,6);

\draw [thick] (7,6)--(7,7)--(8,7)--(8,6);
\draw [->] (8.5,6.5)--(9.5,6.5);
\draw [thick] (10,6)--(11,6);
\end{tikzpicture}
\end{center}
in conjunction with the fact that $\gr(\J_R)=3$ and $\J_R$ has a subgraph induced by $E\oplus\{0,\pm\e_2\}\setminus\{0\}$ isomorphic to $\J_{E\oplus\Z_3}$, we can obtain cycles of arbitrary lengths $l$, for all $3\leq l\leq|V(\J_R)|$. Hence $\J_R$ is pancyclic.

Now, let $R$ be a finite non-local ring such that $R\not\cong\Z_2\oplus F$ ($F$ field). First, suppose that $R/J(R)$ is Hamiltonian. Then, since every cycle $\C$ in $\J_{R/J(R)}$ induces a cycle in $\J_R$ by considering representatives of elements of $\C$, $\J_R$ contains cycles of every length $l$ with $3\leq l\leq|V(\J_{R/J(R)})|$. To show that $\J_R$ is pancyclic, let $\C$ be a Hamiltonian cycle in $\J_{R/J(R)}$ and $\C'$ be the set of representatives of elements of $\C$. Then, by \cite[Lemma 2.1]{aa-ae-mfdg} and the fact that $\J_R$ has a complete subgraph $1+J(R)$, one can use $\C'$ to construct cycles of every length $l$, for $|V(\J_{R/J(R)})|<l\leq|V(\J_R)|$. Therefore $\J_R$ is pancyclic.

Finally, suppose that $R/J(R)$ is not Hamiltonian. Thus $R/J(R)\cong\Z_2\oplus F$ such that $F$ is a finite field and $J(R)\neq0$. But then, by Figure 9, $\J_R$ is pancyclic and the proof is complete.
\end{proof}
\begin{corollary}
A Jacobson graph is Hamiltonian if and only if it is pancyclic.
\end{corollary}

In spite of the fact that almost all Jacobson graphs are Hamiltonian, Eulerian Jacobson graphs are few. To show this, we first obtain the degree of a given vertex in a Jacobson graph.
\begin{lemma}\label{degree}
Let $R=R_1\oplus\cdots\oplus R_n$ be the decomposition of the finite ring $R$ into local rings $R_i$ with associated fields $F_i$, respectively. If $x=(x_1,\ldots,x_n)\in V(\J_R)$, then
\[\deg(x)=|R|\left(1-\prod_{x_i\notin J(R_i)}\left(1-\frac{1}{|F_i|}\right)\right)-\varepsilon_x,\]
where $\varepsilon_x=1$ if $x\in\prod_{i=1}^n(J(R_i)+\{0,\pm1\})$ and $\varepsilon_x=0$ otherwise.
\end{lemma}
\begin{proof}
Let $X=\{i:x_i\notin J(R_i)\}$ and $A_i=\{y\in V(\J_R):y_i\in x_i^{-1}+J(R_i)\}$ for all $i\in X$. Then for every $i\in X$, $|A_i|=|R|/|F_i|$ and
\[\deg(x)=\left|\bigcup_{i\in X}A_i\right|-\varepsilon_x.\]
Using the inclusion-exclusion principal the result follows easily.
\end{proof}

The following result is an immediate consequence of the above lemma.
\begin{corollary}\label{edges}
Let $R=R_1\oplus\cdots\oplus R_n$ be the decomposition of the finite ring $R$ into local rings $R_i$ with associated fields $F_i$, respectively. Then
\[2|E(\J_R)|=|R|^2\left(1-\prod_{i=1}^n\left(1-\frac{1}{|F_i|}+\frac{1}{|F_i|^2}\right)\right)-|J(R)|(3^\mathcal{O}2^\mathcal{E}-1),\]
where $\mathcal{O}$ and $\mathcal{E}$ denote the number of indices $i$ such that $|F_i|$ is odd and even, respectively.
\end{corollary}

Now we obtain all finite rings with an Eulerian tour or trail.
\begin{theorem}\label{Eulerian}
Let $R$ be a finite ring. Then $\J_R$ is Eulerian if and only if $R\cong\Z_2$, or $|R|$ is odd and $R/J(R)\cong\Z_3\oplus\cdots\oplus\Z_3$ is a direct sum of at least two copies of $\Z_3$.
\end{theorem}
\begin{proof}
Let $R=R_1\oplus\cdots\oplus R_n$ be the decomposition of the finite ring $R$ with Eulerian Jacobson graph into local rings $R_i$ with associated fields $F_i$, respectively. If $R=R_1$ is a local ring, then by \cite[Theorem 2.2]{aa-ae-mfdg} and the fact that $\J_R$ is connected, $F_1\cong\Z_2$ and $\J_R$ is a complete graph. Thus $\deg(1)=|J(R)|-1$ is even, that is, $|J(R)|$ is odd, and by invoking the fact that a finite local ring has prime power order, it follows that $|J(R)|=1$. Hence $R\cong\Z_2$.

Now suppose that $n>1$. We first observe, by Lemma \ref{degree}, that $\deg(\e_i)=|R|/|F_i|-1$ is even and hence $|R|/|F_i|$ is odd for all $i=1,\ldots,n$. Hence $|R|$ is odd and by using Lemma \ref{degree} once more, $|F_i|=3$ for all $i=1,\ldots,n$. The converse follows directly from Lemma \ref{degree}.
\end{proof}
\begin{theorem}\label{Euleriantrail}
Let $R$ be a finite ring, which is not Eulerian. Then $\J_R$ contains an Eulerian trail if and only if $R\cong\Z_4$, $\Z_2[x]/(x^2)$, $\Z_2\oplus\Z_2$ or $\Z_2[x]/(x^2+x+1)\oplus\Z_2$.
\end{theorem}
\begin{proof}
Let $R=R_1\oplus\cdots\oplus R_n$ be a decomposition of $R$ into local rings $R_i$ with associated fields $F_i$. Suppose that $\J(R)$ has an Eulerian trail and hence exactly two vertices with odd degrees. If $R=R_1$ is a local ring, then clearly $|F_1|=|J(R)|=2$, from which it follows that $R\cong\Z_4$ or $Z_2[x]/(x^2)$. Now, assume that $R$ is not a local ring. Since the elements in any right coset of $J(R)$ as an additive subgroup of $R$ have the same degrees, we should have $|J(R)|\leq2$. If $|J(R)|=2$, then $\e_1,\e_2,\1$ have odd degrees contradicting our assumption. Hence $J(R)=0$. If $n\geq3$, then since at most two vertices among $\e_1,\e_2,\e_3,\1$ have odd degrees, we conclude that $|R|$ is odd. Then by Lemma \ref{degree}, $\deg(x)\equiv1+\varepsilon_x\pmod2$ for every $x\in R\setminus J(R)$. Hence $|F_i|\geq4$ for some $i$ and we may assume that $i=1$. If $a\in F_1\setminus\{0,\pm1\}$, then $a\e_1,a\e_1+\e_2,a\e_1+\e_3$ have odd degrees, which is a contradiction. Therefore $n=2$ and $R=F_1\oplus F_2$. If $|F_1|\not\equiv|F_2|\pmod2$, say $|F_1|$ is odd and $|F_2|$ is even, then 
$\pm\e_1,\1$ have odd degrees, which is a contradiction. Thus $|F_1|\equiv|F_2|\pmod2$. Since by Lemma \ref{degree}, $\deg(x)\equiv1+\varepsilon_x\pmod2$ for every $x\in R\setminus J(R)$, it follows that $|F_1|,|F_2|\leq5$. Now, a simple verification shows that $(|F_1|,|F_2|)=(2,2)$, $(4,2)$ or $(2,4)$ so that $R\cong\Z_2\oplus\Z_2$ or $\Z_2[x]/(x^2+x+1)\oplus\Z_2$, as required.
\end{proof}
\section{Induced cycles and paths}
In this section, we shall compute the length of the longest induced cycles and paths in Jacobson graphs. As it is mentioned before, induced cycles and paths are important in studying other graph theoretical concepts. We begin with calculating the length of the longest induced cycles and then we use the same arguments to find the length of the longest induced paths.

Let $l_c(\J_R)$ denote the length of the longest induced cycles in $\J_R$. If $(R,\m)$ is a finite local ring with associated field $F$, then by \cite[Theorem 2.2]{aa-ae-mfdg},
\[l_c(\J_R)=\begin{cases}\infty,&\m=0,\\\infty,&|\m|=2,|F|\leq3,\\3,&|\m|\geq3,|F|\leq3,\\4,&\m\neq0,|F|\geq4.\end{cases}\]

Now, we obtain the length of the longest induced cycles of finite non-local rings. For this, let $\varepsilon_F=0$ if $|F|$ is even and $\varepsilon_F=1$ if $F$ is odd for any finite field $F$.
\begin{theorem}\label{maximalinducedcycle}
Let $R=R_1\oplus\cdots\oplus R_n$ ($n>1$) be the decomposition of a finite ring $R$ into finite local rings $R_i$ with associated fields $F_i$, respectively. If $R\not\cong\Z_2\oplus\Z_2$, then $l_c(\J_R)=3$ if $J(R)\neq0$ and $R/J(R)\cong\Z_2\oplus\Z_2$, and
\[l_c(\J_R)=|F_1|+\cdots+|F_{n-1}|-(n-1)+\theta^*,\]
where
\[\theta^*=2\theta-\left\lfloor\frac{4}{|F_n|+4-2\theta-\varepsilon_{F_n}}\right\rfloor\]
and
\[\theta=\frac{1}{2}\min\left\{|F_n|+\varepsilon_{F_n},\sum_{i=1}^{n-1}(|F_i|+\varepsilon_{F_i})\right\}\]
otherwise.
\end{theorem}
\begin{proof}
First we show that $l_c(\J_R)=l_c(\J_{R/J(R)})$ whenever $J(R)=0$ or $R/J(R)\not\cong\Z_2\oplus\Z_2$. Clearly $l_c(\J_{R/J(R)})\leq l_c(\J_R)$. Suppose on the contrary that $l_c(\J_{R/J(R)})<l_c(\J_R)$. Then, in a longest induced cycle $\C$ of $\J_R$, there exist two vertices $x,y\in V(\C)$ belonging to the same coset of $J(R)$ that is $x+J(R)=y+J(R)$. By \cite[Lemma 2.1]{aa-ae-mfdg}, we can easily see that $|\C|\leq4$. Hence $l_c(\J_{R/J(R)})\leq3$ and by \cite[Theorem 4.6]{aa-ae-mfdg}, it follows that $J(R)\neq0$ and $R/J(R)\cong\Z_2\oplus\Z_2$, a contradiction. 

If $l_c(\J_R)=3$, then by \cite[Theorem 4.6]{aa-ae-mfdg}, $R/J(R)\cong\Z_2\oplus\Z_2$, $\Z_2\oplus\Z_3$ or $\Z_2\oplus\Z_2\oplus\Z_2$ and we are done. 

Similarly, if $l_c(\J_R)=4$, then again by \cite[Theorem 4.6]{aa-ae-mfdg}, $R/J(R)\cong\Z_2\oplus F$ ($F$ is a finite field with $|F|\geq4$), $\Z_3\oplus\Z_3$, $\Z_2\oplus\Z_2\oplus\Z_3$ or $\Z_2\oplus\Z_2\oplus\Z_2\oplus\Z_2$ and the result follows.

Now, suppose that $l_c(\J_R)\geq5$. By the preceding paragraph, we can assume without loss of generality that $J(R)=0$ and hence $R=F_1\oplus\cdots\oplus F_n$. We first show that 
\[l_c(\J_R)\leq|F_1|+\cdots+|F_{n-1}|-(n-1)+\theta^*,\]
where $\theta^*$ is given in the theorem. For this let $c(F_i)=(|F_i|+\varepsilon_{F_i})/2$ be the number of connected components of $\J_{F_i}$. Then, the number $c_{bp}(F_i)$ of bipartite components (edges) of $\J_{F_i}$ equals $(|F_i|-2-\varepsilon_{F_i})/2$. Let $\C$ be a longest induced cycle in $\J_R$ with maximum number of zeros on its elements coordinates. If $\pi_i$ denotes the projection on the $i$th coordinate, then for every $x,y,z\in V(\C)$ such that $x\sim y\sim z$ and $\pi_i(y)\neq0$, we have either $\pi_i(x)=\pi_i(y)^{-1}$ or $\pi_i(z)=\pi_i(y)^{-1}$ for otherwise by setting $\pi_i(y)=0$, we reach to an induce cycle with the same length as $\C$ and more zeros than $\C$, contradicting our assumption. Hence for every $a\in F_i\setminus\{0\}$, either $a$ and $a^{-1}$ does not appear in the $i$th coordinates of vertices of $\C$, or they constitute a consecutive sequence of the form 
\[\{a,a^{-1}\},\ \ \  \{a^{-1},a\}, \ \ \ \{a,a^{-1},a\},\ \ \ or\ \ \ \{a^{-1},a,a^{-1}\},\]
in which the last two sequences appear only if $a\neq\pm1$. Moreover, $a$ contributes at most two edges to the cycle and it contributes two edges only if $a\neq\pm1$. Now suppose that $\pi_n(\C)\setminus\{0\}=\{a_1^{\pm1},\ldots,a_m^{\pm1}\}$. Let $f:E(\C)\longrightarrow\{0,1\}$ be an edge labeling of $\C$ defined by
\[f(\{x,y\})=\begin{cases}1,&\pi_n(x)\pi_n(y)=1,\\0,&\mbox{otherwise}.\end{cases}\]
This labeling partitions the edges of $\C$ into $2m$ sets of consecutive edges, alternately labeled by zeros and ones. Let $\P_1,\ldots,\P_m$ and $\Q_1,\ldots,Q_m$ be the sets of consecutive edges labeled with zeros and ones, respectively. Since $\pi_n(x)\pi_n(y)\neq1$ for each edge $\{x,y\}\in\P_k$ there exists $i<n$ such that $\pi_i(x)\pi_i(y)=1$, that is, $x,y$ are connected via their $i$th coordinate. By the aforementioned discussions the number of such connections are at most 
\[(|F_1|-1)+\cdots+(|F_{n-1}|-1)\]
and hence
\[|E(\P_1)|+\cdots+|E(\P_m)|\leq|F_1|+\cdots+|F_{n-1}|-(n-1).\]
With the same reason, the number $m$ is at most
\[c(F_1)+\cdots+c(F_{n-1}).\]
On the other hand, $m\leq c(F_n)$. Hence
\[m\leq\min\{c(F_n),c(F_1)+\cdots+c(F_{n-1})\}.\]
Now since $|E(\Q_i)|\leq2$ and the number of $Q_i$ with $|E(\Q_i)|=2$ is at most the number of bipartite components of $\J_{F_n}$, it follows that
\[|E(\Q_1)|+\cdots+|E(\Q_m)|\leq 2m-\max\left\{0,m-c_{bp}(F_n)\right\}.\]
As $m-c_{bp}(F_n)\leq2$ and
\[\left\lfloor\frac{2}{3-x}\right\rfloor=\max\{0,x\},\]
whenever $x\leq2$, it follows that
\begin{align*}
|E(\Q_1)|+\cdots+|E(\Q_m)|&\leq2m-\left\lfloor\frac{2}{3-c_{bp}(F_n)}\right\rfloor\\
&=2m-\left\lfloor\frac{4}{|F_n|+4-2m-\varepsilon_{F_n}}\right\rfloor.
\end{align*}
Therefore
\begin{align*}
l_c(\J_R)=|\C|&=|E(\P_1)|+\cdots+|E(\P_m)|+|E(\Q_1)|+\cdots+|E(\Q_m)|\\
&\leq|F_1|+\cdots+|F_{n-1}|-(n-1)+\theta^*,
\end{align*}
as required. According to the above discussions it is straightforward to construct an induced cycle, whose length is equal to the given upper bound. The proof is complete.
\end{proof}

The same as for cycles, let $l_p(\J_R)$ denoted the length of the longest induced paths in $\J_R$. If $(R,\m)$ is a finite local ring with associated field $F$, then
\[l_p(\J_R)=\begin{cases}0,&\m=0,|F|\leq3,\\1,&\m=0,|F|\geq4,\\1,&\m\neq0,|F|\leq3,\\2,&\m\neq0,|F|\geq4.\end{cases}\]

For a finite non-local ring, the length of the longest induced paths is given by the following theorem.
\begin{theorem}\label{maximalinducedpath}
Let $R$ be finite non-local ring. Then $l_p(\J_R)=l_c(\J_R)$ except when $R\cong\Z_2\oplus\Z_2$ or $R/J(R)\cong\Z_2\oplus F$ with $|F|\geq7$. Moreover, $l_p(\Z_2\oplus\Z_2)=2$ and $l_p(\Z_2\oplus F)=5$ whenever $|F|\geq7$.
\end{theorem}
\begin{proof}
First we show that $l_p(\J_R)=l_p(\J_{R/J(R)})$. Clearly, $l_p(\J_{R/J(R)})\leq l_p(\J_R)$. Suppose on the contrary that $l_p(\J_{R/J(R)})<l_p(\J_R)$. Then, in a longest induced path $\P$ of $\J_R$, there exist two vertices $x,y\in V(\P)$ belonging to the same coset of $J(R)$, that is, $x+J(R)=y+J(R)$. By \cite[Lemma 2.1]{aa-ae-mfdg}, it is easy to see that $|\P|\leq2$. Hence $l_p(\J_{R/J(R)})\leq1$, which contradicts \cite[Theorem 3.5]{aa-ae-mfdg}.

If $l_c(\J_R)\geq5$, then by using a same method as in the proof of Theorem \ref{maximalinducedcycle}, we can show that $l_p(\J_R)=l_c(\J_R)$. Now suppose that $l_c(\J_R)\leq4$. Then by \cite[Theorem 4.6]{aa-ae-mfdg}, $R/J(R)\cong\Z_2\oplus F$, $\Z_3\oplus\Z_3$, $\Z_2\oplus\Z_2\oplus\Z_2$, $\Z_2\oplus\Z_2\oplus\Z_3$, or $\Z_2\oplus\Z_2\oplus\Z_2\oplus\Z_2$, where $F$ is a finite field. A simple verification shows that $l_p(\J_R)=l_c(\J_R)$ if $R\not\cong\Z_2\oplus F$ ($|F|\geq7$, or $|F|=2$ and $J(R)=0$). Finally, if $R\cong\Z_2\oplus\Z_2$, then $l_p(\J_R)=2$ and if $R\cong\Z_2\oplus F$ with $|F|\geq7$, then $l_p(\J_R)=5$. The proof is complete.
\end{proof}

\end{document}